\documentclass[12pt,twoside,a4paper]{amsart}

\usepackage[textwidth=15cm,textheight=22cm,centering]{geometry}
\usepackage[T1]{fontenc}
\usepackage{lmodern}
\usepackage{microtype}
\usepackage{amsmath,amssymb,amsthm,mathtools}
\usepackage{xurl}
\usepackage[hidelinks]{hyperref}
\newtheorem{theorem}{Theorem}[section]
\newtheorem{maintheorem}{Theorem}

\newtheorem{proposition}[theorem]{Proposition}
\newtheorem{lemma}[theorem]{Lemma}
\newtheorem{Question}[theorem]{Question}
\newtheorem{corollary}[theorem]{Corollary}
\newtheorem{problem}[theorem]{Problem}
\theoremstyle{definition}
\newtheorem{definition}[theorem]{Definition}
\theoremstyle{remark}

\newcommand{\C}{\mathbb C}
\newcommand{\R}{\mathbb R}
\newcommand{\ii}{\sqrt{-1}}
\newcommand{\HH}{\mathsf H}
\DeclareMathOperator{\tr}{tr}
\DeclareMathOperator{\diag}{diag}

\title{Nakano-positive determinants outside the Hodge--Riemann cone}
\author{Zhangchi Chen}
\date{}

\hypersetup{
  pdftitle={Nakano-positive determinants outside the Hodge-Riemann cone},
  pdfauthor={Zhangchi Chen},
  pdfsubject={Counterexamples to Dinh-Nguyen Problem 2.3},
  pdfkeywords={Hodge-Riemann cone, Griffiths cone, Nakano positivity,
  hard Lefschetz theorem, dually Lorentzian polynomials,
  generalized Alexandrov-Fenchel inequality}
}

\address{Zhangchi Chen, School of Mathematical Sciences, Key Laboratory of MEA (Ministry of Education), and Shanghai Key Laboratory of PMMP, East China Normal University, Shanghai 200241, China.}
\email{zcchen@math.ecnu.edu.cn}

\begin{document}
\begin{abstract}
Dinh and Nguy\^en asked whether the determinant of a Griffiths positive matrix of $(1,1)$-forms belongs to the Hodge--Riemann cone.
For every $n\geqslant4$ and $2\leqslant k\leqslant n-2$, we present counterexamples constructed by AI: Nakano positive $k\times k$ matrices of constant $(1,1)$-forms on $\C^n$ whose determinants have singular Lefschetz maps in bidegree $(1,n-k-1)$.
Under the simultaneous diagonalizability (SD) condition, we prove the Hodge--Riemann property in every bidegree $(p,q)$ with $p+q=n-k$ and $\min(p,q)\leqslant1$.
The proof uses the generalized Alexandrov--Fenchel inequality of Ross, S\"u\ss, and Wannerer.
For every $n\geqslant6$ and $2\leqslant k\leqslant n-4$, we also present SD examples whose determinants have singular Lefschetz maps in bidegree $(2,n-k-2)$.
Their underlying idea comes from Ross--Toma's construction.
The remaining cases $n=k\geqslant4$ and $n=k+1\geqslant5$ are open for general Griffiths positive matrices.
They reduce to the top-Chern case of Griffiths' positivity question, equivalently to Finski's double mixed discriminant problem.
We also ask whether a $2\times2$ counterexample on $\C^4$ can be both Nakano positive and dual Nakano positive.
\end{abstract}
\maketitle

\section{Introduction}

Let $X$ be a compact complex manifold of dimension $n$.
A \emph{K\"ahler form} on $X$ is a smooth real closed $(1,1)$-form whose associated Hermitian form is positive definite.
Fix such a form $\omega$.
For $p+q\leqslant n$, put $k=n-p-q$ and $\Omega=\omega^k$, and define the Lefschetz map
\[
 L_{\Omega}^{p,q}\colon
 H^{p,q}(X,\C)\longrightarrow H^{n-q,n-p}(X,\C),
 \qquad [\alpha]\longmapsto[\alpha\wedge\Omega].
\]
The classical hard Lefschetz theorem says that $L_{\Omega}^{p,q}$ is an isomorphism.
To state the two companion theorems, following Dinh and Nguy\^en \cite[p.~121]{DN},
define the \emph{primitive subspace associated with $\Omega$}, that is, the space of $\Omega$-primitive $(p,q)$-classes, by
\[
 P^{p,q}_{\Omega}
 =\bigl\{[\alpha]\in H^{p,q}(X,\C):
 [\alpha\wedge\Omega\wedge\omega]=0\bigr\},
\]
and define the Hermitian form
\[
 Q_{\Omega}([\alpha],[\beta])
 =\ii^{p-q}(-1)^{(p+q)(p+q-1)/2}
 \int_X\alpha\wedge\overline\beta\wedge\Omega.
\]
The Hodge--Riemann bilinear relations assert that $Q_{\Omega}$ is positive definite on $P^{p,q}_{\Omega}$, while the Lefschetz decomposition theorem gives the $Q_{\Omega}$-orthogonal direct sum
\[
 H^{p,q}(X,\C)
 =[\omega]\wedge H^{p-1,q-1}(X,\C)\oplus P^{p,q}_{\Omega}.
\]
We refer to \cite[Sections~6.2--6.3]{Voisin} for these classical statements.

Mixed versions of the three classical theorems remain valid for products of K\"ahler forms:
Timorin proved the linear case \cite[Proposition~1, the Main Theorem, and Corollary~2]{Timorin},
and Dinh and Nguy\^en proved the corresponding results on compact K\"ahler manifolds \cite[Theorems~A--C]{DNmixed}.
Dinh and Nguy\^en then introduced the pointwise \emph{Hodge--Riemann cone}: roughly, it is the set of real $(k,k)$-forms that can be deformed to $\omega^k$ while the necessary Lefschetz maps remain invertible.
The precise linear definition is recalled in Section~\ref{sec:HR}.
Their local-to-global theorem states that a smooth closed $(k,k)$-form whose values lie pointwise in the Hodge--Riemann cone for $(p,q)$ gives a cohomology class satisfying all three theorems in that bidegree \cite[Theorem~1.1]{DN}.

This reduces the search for new cohomology classes with the three properties to a problem in exterior algebra.
Let $V$ be a complex vector space of dimension $n$, and put
\[
 V^{p,q}=\Lambda^pV^*\otimes\Lambda^q\overline{V^*}.
\]
For $\Omega\in V^{k,k}$ and $p+q=n-k$, the pointwise Lefschetz map is
\[
 L_{\Omega}^{p,q}\colon V^{p,q}\longrightarrow V^{n-q,n-p},
 \qquad \alpha\longmapsto\alpha\wedge\Omega.
\]
Now let $M=(\alpha_{ab})_{a,b=1}^k$ be a Hermitian matrix of constant $(1,1)$-forms on $V$.
It is \emph{Griffiths positive} when
\[
 \sum_{a,b=1}^k\theta_a\overline{\theta_b}\,\alpha_{ab}
\]
is a K\"ahler form for every $0\ne\theta\in\C^k$.
The \emph{Griffiths cone} consists of the determinant forms $\det M$ obtained in this way.
Dinh and Nguy\^en asked \cite[Problem~2.3]{DN}:

\begin{problem}[Dinh--Nguy\^en]\label{prob:DN}
Is the Griffiths cone contained in the Hodge--Riemann cone?
\end{problem}

The answer is affirmative when the matrix itself is diagonal:
\[
 M=\diag(\omega_1,\ldots,\omega_k).
\]
Griffiths positivity then says precisely that every $\omega_j$ is K\"ahler, and
\[
 \det M=\omega_1\wedge\cdots\wedge\omega_k.
\]
Thus Timorin's mixed Hodge--Riemann theorem, recorded by Dinh and Nguy\^en as \cite[Proposition~2.2]{DN}, gives an affirmative answer in this case.
The problem is whether the same conclusion survives the off-diagonal entries of a general Griffiths positive matrix.
Using Griffiths's positivity result for the second Chern form \cite{Griffiths}, one can prove the affirmative cases $(n,k)=(2,2),(3,2)$ \cite[Corollary~1.8]{Chen}.
More recently, Wan's rank-three positivity theorem gives the affirmative case $(n,k)=(3,3)$ \cite[Theorem~0.3]{Wan}, and restriction to three-dimensional subspaces gives $(n,k)=(4,3)$.

We first present counterexamples constructed by AI, giving a negative answer throughout the range $n\geqslant4$ and $2\leqslant k\leqslant n-2$, even under the stronger assumption of Nakano positivity.
For the Hermitian coefficient tensor of $M$, Nakano positivity means positivity on the full tensor product $\C^k\otimes\C^n$, whereas Griffiths positivity tests only decomposable tensors.
The precise definitions are recalled in Section~\ref{sec:positivity}.

\begin{maintheorem}\label{thm:main}
For every $n\geqslant4$ and $2\leqslant k\leqslant n-2$, there is a Nakano positive Hermitian $k\times k$ matrix $M$ of constant $(1,1)$-forms on $\C^n$ such that
\[
 L_{\det M}^{1,n-k-1}\text{ is not an isomorphism}.
\]
In particular, $\det M$ belongs to the Griffiths cone but not to the Hodge--Riemann cone.
\end{maintheorem}

The proof gives an explicit nonzero kernel and an explicit parameter in every rank.
For $n=k+2$ the obstruction is already in bidegree $(1,1)$; an extension in the coordinate directions gives all larger dimensions.
Transposition preserves the determinant and exchanges Nakano positivity with dual Nakano positivity.
Thus Theorem~\ref{thm:main} also gives dual Nakano positive counterexamples.
The matrices presented in its proof do not satisfy both conditions, since their partially transposed coefficient tensors have a negative eigenvalue, as computed in Section~\ref{sec:unrestricted}.
For $n=4$ and $k=2$, it remains open whether $L_{\det M}^{1,1}$ can be singular when $M$ satisfies both positivity conditions.
This is Question~\ref{qu:simultaneous-nakano} below.

Theorem~\ref{thm:main} leads us to ask which additional conditions on $M$ imply the Lefschetz and Hodge--Riemann theorems.
In \cite[Section~4]{Chen}, the author introduced the following additional condition for investigating this question.
\begin{definition}[Simultaneously diagonalizable (SD) condition]
The matrix $M$ is \emph{simultaneously diagonalizable} if there are complex linear coordinates such that
\begin{equation}\label{eq:sd}
 M=\sum_{j=1}^n B_jV_j,\qquad
 V_j=\frac{\ii}{2}dz_j\wedge d\bar z_j,\qquad B_j=B_j^*.
\end{equation}
\end{definition}
Note that the Hermitian coefficient matrices $B_j$ need not commute.
Under the SD condition, Griffiths positivity is equivalent to positive definiteness of every $B_j$ \cite[Proposition~4.1]{Chen}.
Under this condition, Chen proved the affirmative cases $n=4,5$ with $k=2$, as well as the Lefschetz isomorphism $L_{\det M}^{1,n-3}$ for every $n\geqslant4$ \cite[Theorem~1.9 and Corollary~4.8]{Chen}.
The same work also gives an affirmative answer under SD in the boundary cases $k=n-1,n$ for every rank \cite[Corollary~4.5]{Chen}.

For the entire SD class, it is enough to establish the Lefschetz isomorphism in bidegree $(1,n-k-1)$:
complex conjugation gives the isomorphism in bidegree $(n-k-1,1)$,
while positivity gives the pure-bidegree isomorphisms for $\det M$ in bidegrees $(0,n-k)$ and $(n-k,0)$ and for $\det M\wedge\omega^2$ in bidegrees $(0,n-k-2)$ and $(n-k-2,0)$.
Lemma~\ref{lem:sd-reduction} derives the Hodge--Riemann property in all bidegrees with $\min(p,q)\leqslant1$ from these isomorphisms.

\begin{maintheorem}\label{thm:sd}
Let $n\geqslant4$ and $2\leqslant k\leqslant n-2$, and let $M$ be a Griffiths positive Hermitian $k\times k$ matrix of constant $(1,1)$-forms satisfying SD.
Then $L_{\det M}^{1,n-k-1}$ is an isomorphism.
Moreover, for every $p+q=n-k$ with $\min(p,q)\leqslant1$, the form $\det M$ is Hodge--Riemann for $(p,q)$ with respect to any positive reference $(1,1)$-form.
\end{maintheorem}

\begin{corollary}\label{cor:sd-complete}
Under the hypotheses of Theorem~\ref{thm:sd}, if $n-k=2$ or $3$, then $\det M$ is Hodge--Riemann in every bidegree $p+q=n-k$.
\end{corollary}
\begin{proof}
Every such bidegree has $\min(p,q)\leqslant1$.
\end{proof}

The key input in Theorem~\ref{thm:sd} is the theory of dually Lorentzian polynomials developed by Ross, S\"u\ss, and Wannerer, in particular their generalized Alexandrov--Fenchel inequality and its equality characterization \cite[Theorem~8.2]{RSW}.
We construct an auxiliary multiaffine determinantal polynomial whose stability, together with that of its dual, allows us to apply the Lorentzian criterion of Br\"and\'en and Huh \cite[Proposition~2.2]{BH}.
The generalized inequality then gives a real symmetric matrix of signature $(1,k+1)$, which is precisely a nontrivial block of $L_{\det M}^{1,n-k-1}$.
Its strict equality characterization is essential: the Lorentzian property alone does not exclude a singular Hessian.

In the next bidegree, an additional block can be singular.

\begin{maintheorem}\label{thm:sdnegative}
For every $n\geqslant6$ and $2\leqslant k\leqslant n-4$, there exists a Griffiths positive Hermitian $k\times k$ matrix $M$ of constant $(1,1)$-forms satisfying SD such that $L_{\det M}^{2,n-k-2}$ is singular.
In particular, $\det M$ does not belong to the Hodge--Riemann cone.
\end{maintheorem}

For these same matrices, Theorem~\ref{thm:sd} still guarantees the Hodge--Riemann property in bidegree $(1,n-k-1)$.
Thus SD recovers the Hodge--Riemann property in the first mixed bidegree in every rank, but it does not suffice for all bidegrees when $n-k\geqslant4$.

For Theorem~\ref{thm:sdnegative}, we present examples whose underlying idea comes from Ross--Toma's construction \cite[Example~9.2, p.~37 of arXiv:1905.13636v3]{RT}.
The six-dimensional seed first appeared there as a second Chern class of an ample real-twisted rank-three bundle on $(\mathbf P^2)^3$.
The examples presented here give a rank-two determinant realization of that seed, after positive rescaling and a change of parameter, and extensions in rank and dimension.
Theorem~\ref{thm:sdnegative} was also obtained using AI.
Section~\ref{sec:sdnegative} gives the exact comparison with Ross--Toma.

Enhan Li, of the University of Science and Technology of China (USTC), independently found a counterexample to Dinh--Nguy\^en's question for $k=2$ and every $n\geqslant4$.

For general Griffiths positive matrices, the remaining cases of Dinh--Nguy\^en's question are $n=k\geqslant4$ and $n=k+1\geqslant5$.
In these boundary degrees the question is equivalent to strict weak positivity of the determinant: the top-Chern case of Griffiths' positivity question, or Finski's double mixed discriminant problem.
Under either Nakano positivity or dual Nakano positivity, these boundary cases are affirmative in every rank.
Section~\ref{sec:scope} explains the equivalence and the known cases.
Question~\ref{qu:griffiths-determinant} states the determinant positivity problem for $k\geqslant4$.

The paper is organized as follows.
Section~\ref{sec:preliminaries} gives the definitions and the deformation and extension arguments.
Section~\ref{sec:unrestricted} proves Theorem~\ref{thm:main}.
Section~\ref{sec:sdpositive} develops the dually Lorentzian polynomial argument and proves Theorem~\ref{thm:sd}.
Section~\ref{sec:sdnegative} proves Theorem~\ref{thm:sdnegative}.
Section~\ref{sec:scope} formulates two open questions, and Appendix~\ref{sec:verification} contains exact Python verification programs.

\section{Preliminaries}\label{sec:preliminaries}

\subsection{Lefschetz and Hodge--Riemann forms}\label{sec:HR}

Let $V$ be a complex vector space of dimension $n$.
In the linear setting, a real $(1,1)$-form $\omega$ is K\"ahler if it is positive definite, equivalently if some complex coordinates satisfy
\[
 \omega=\frac{\ii}{2}\sum_{j=1}^n
 dz_j\wedge d\bar z_j.
\]
Fix such an $\omega$.
We now recall the pointwise Hodge--Riemann definition using the Lefschetz maps introduced above.

\begin{definition}[Dinh--Nguy\^en {\cite[Definition~2.1]{DN}}]
\label{def:HR-cone}
Let $p,q,k$ be nonnegative integers with $p+q=n-k$.
A real $(k,k)$-form $\Omega$ is \emph{Lefschetz for $(p,q)$} if $L_{\Omega}^{p,q}$ is an isomorphism.
A real $(k,k)$-form $\Omega$ is \emph{Hodge--Riemann for $(p,q)$} if there is a continuous path of real forms $\Omega_t$, $0\leq t\leq1$, with $\Omega_0=\Omega$ and $\Omega_1=\omega^k$, such that
\[
 \Omega_t\wedge\omega^{2r}
 \quad\text{is Lefschetz for }(p-r,q-r)
\]
for every integer $0\leq r\leq\min\{p,q\}$ and every $0\leq t\leq1$.
The corresponding connected component is the Hodge--Riemann cone for $(p,q)$.
The form is called Hodge--Riemann if it has this property for every $p+q=n-k$.
\end{definition}

In particular, a singular $L_\Omega^{p,q}$ excludes the Hodge--Riemann property, regardless of the reference form $\omega$: the defining condition would already fail at $t=0$ and $r=0$.

\subsection{Griffiths, Nakano, and dual Nakano positivity}\label{sec:positivity}

Write a Hermitian matrix of $(1,1)$-forms as
\begin{equation}\label{eq:coefficient-tensor}
 \alpha_{ab}=\frac{\ii}{2}
 \sum_{\mu,\nu=1}^n
 A_{a\bar b\mu\bar\nu}\,
 dz_\mu\wedge d\bar z_\nu.
\end{equation}
The coefficients define a Hermitian form $A$ on $\C^k\otimes\C^n$.
The Griffiths condition stated in the introduction is equivalently
\[
 A(\theta\otimes\xi,\theta\otimes\xi)>0
 \quad(0\ne\theta\in\C^k,\ 0\ne\xi\in\C^n).
\]
The matrix $M$ is Nakano positive precisely when
\[
 \sum_{a,b,\mu,\nu}
 A_{a\bar b\mu\bar\nu}u_{a\mu}\overline{u_{b\nu}}>0
 \quad(0\ne u\in\C^k\otimes\C^n).
\]
These are the standard algebraic Griffiths and Nakano conditions; see Griffiths \cite[p.~185, (0.1)]{Griffiths} and Nakano \cite[p.~8, (2.10)]{Nakano}.

The matrix $M$ is \emph{dual Nakano positive} if its transpose $M^{\mathsf T}$ is Nakano positive, that is, if
\[
 \sum_{a,b,\mu,\nu}
 A_{b\bar a\mu\bar\nu}u_{a\mu}\overline{u_{b\nu}}>0
 \quad(0\ne u\in\C^k\otimes\C^n).
\]
Thus dual Nakano positivity tests the partial transpose of the coefficient tensor in the matrix indices; see \cite[Section~2.1]{Finski}.

Testing the Nakano inequality on $u=\theta\otimes\xi$ shows that Nakano positivity implies Griffiths positivity.
Consequently, the determinant of a Nakano positive matrix belongs to the Griffiths cone.

Because $(1,1)$-forms have even total degree, their wedge products commute.
Thus the ordinary permutation formula defines $\det M$.
For Hermitian $M$, the determinant is a real $(k,k)$-form.
In particular, $\det M^{\mathsf T}=\det M$, so transposition preserves every Lefschetz map of the determinant and exchanges the two Nakano positivity conditions.

We use $I_d$ for the $d\times d$ identity matrix and set
\[
 E_{ij}=\frac{\ii}{2}dz_i\wedge d\bar z_j,\qquad V_j=E_{jj},
 \qquad \omega_0=\sum_{j=1}^n V_j.
\]
The form $\omega_0$ is the standard positive form in the chosen coordinates; $\omega$ continues to denote an arbitrary positive reference form.
Write $[n]=\{1,\ldots,n\}$.
For $I\subset[n]$ and indeterminates $x_1,\ldots,x_n$, put
\[
 V_I=\prod_{i\in I}V_i,\qquad x_I=\prod_{i\in I}x_i.
\]
Products of the $V_i$ are computed in $\R[V_1,\ldots,V_n]/(V_1^2,\ldots,V_n^2)$.
We write $\binom{J}{r}=\{R\subset J:|R|=r\}$.
For index sets, $dz_I$ and $d\bar z_I$ denote the corresponding wedge products in increasing order.
Empty products equal $1$.
For a complex matrix $U$, $U^*$ denotes its conjugate transpose and $U^{\mathsf T}$ its transpose.
The signature of a nonsingular real symmetric or Hermitian matrix is the pair consisting of its numbers of positive and negative eigenvalues.

\subsection{SD coefficients and mixed discriminants}

Write $V_j=(\ii/2)\,dz_j\wedge d\bar z_j$.
The simultaneous diagonalizability condition means that
\[
 M=\sum_{j=1}^n B_jV_j,
 \qquad B_j=B_j^*.
\]
It concerns the form coordinates; the coefficient matrices $B_j$ need not commute.
Under SD one has
\[
\begin{aligned}
 M\text{ Griffiths positive}
 &\Longleftrightarrow B_j>0\text{ for every }j\\
 &\Longleftrightarrow M\text{ Nakano positive}\\
 &\Longleftrightarrow M\text{ dual Nakano positive}.
\end{aligned}
\]
The first equivalence is Chen \cite[Proposition~4.1]{Chen}.
For the other two, order the tensor-product basis by the form index: the Nakano and dual Nakano coefficient matrices are block diagonal with blocks $B_j$ and $B_j^{\mathsf T}$, respectively.
Assuming these equivalent positivity conditions, put
\[
 \Omega=\det M=\sum_{|I|=k}d_IV_I,
 \qquad d_I=k!D_k(B_i:i\in I)>0.
\]
Here $D_k$ is the mixed discriminant, the symmetric polarization of the determinant studied by Alexandrov \cite{Alexandrov}, with normalization
\[
 D_k(A_1,\ldots,A_k)=\frac1{k!}
 \left.
 \frac{\partial^k}{\partial t_1\cdots\partial t_k}
 \det\left(\sum_{\ell=1}^k t_\ell A_\ell\right)
 \right|_{t=0}.
\]
In particular, $D_k(A,\ldots,A)=\det A$.
Weizhe Zheng proved the bounds, recorded in Chen \cite[Theorem~4.3]{Chen},
\[
 \prod_{\ell=1}^k\lambda_{\min}(A_\ell)
 \leqslant D_k(A_1,\ldots,A_k)
 \leqslant \prod_{\ell=1}^k\lambda_{\max}(A_\ell)
\]
for positive definite Hermitian matrices.
They give $d_I>0$, the positivity used below.

\subsection{From Lefschetz isomorphisms to Hodge--Riemann forms}

Definition~\ref{def:HR-cone} requires more than invertibility at a single form: the relevant Lefschetz maps must remain invertible along a path to the power of a reference form.
A related deformation argument for primitive definiteness in bidegree $(1,n-k-1)$ appears in Berndtsson and Sibony \cite[Proposition~9.1]{BS}.
For SD matrices, Chen's deformation argument \cite[Section~2, following Question~2.4]{Chen} provides such a path once the mixed Lefschetz isomorphism is known throughout the class.
The following lemma makes this reduction precise and allows an arbitrary positive reference form, which need not be diagonal in the coordinates of~\eqref{eq:sd}.

\begin{lemma}\label{lem:sd-reduction}
Let $n\geqslant4$, $2\leqslant k\leqslant n-2$, and let $M=\sum_jB_jV_j$, where every $B_j$ is positive definite.
With the standard form $\omega_0=\sum_jV_j$, set
\[
 M_t=(1-t)M+t\omega_0 I_k,\qquad \Omega_t=\det M_t,
 \qquad 0\leqslant t\leqslant1.
\]
If $L_{\Omega_t}^{1,n-k-1}$ is an isomorphism for every $t$, then $\det M$ is Hodge--Riemann in every bidegree $p+q=n-k$ with $\min(p,q)\leqslant1$, relative to any positive reference form $\omega$.
Consequently, a proof of these Lefschetz isomorphisms for all SD matrices with positive definite coefficient matrices implies the corresponding Hodge--Riemann assertion for the entire class.
\end{lemma}

\begin{proof}
Every coefficient matrix $(1-t)B_j+tI_k$ is positive definite, so $\Omega_t$ has strictly positive coefficients in the basis $V_I$.
Thus the pure-bidegree maps $L_{\Omega_t}^{0,n-k}$ and its conjugate are isomorphisms: their associated Hermitian pairings are diagonal and positive definite.
The assumed mixed map and its conjugate give all remaining bidegrees with $\min(p,q)\leqslant1$.

For $\min(p,q)=1$, the cone definition also requires the pure-bidegree maps of $\Omega_t\wedge\omega^2$.
Choose $a_t,b>0$ such that $\Omega_t-a_t\omega_0^k$ is a nonnegative linear combination of the $V_I$ and $\omega-b\omega_0$ is positive semidefinite.
Then
\begin{align*}
 \Omega_t\wedge\omega^2-a_tb^2\omega_0^{k+2}
 &=(\Omega_t-a_t\omega_0^k)\wedge\omega^2\\
 &\quad+a_t\omega_0^k\wedge(\omega-b\omega_0)\wedge(\omega+b\omega_0)
\end{align*}
is a nonnegative linear combination of products of positive semidefinite $(1,1)$-forms.
Writing their factors as sums of rank-one forms shows that their associated pairings on pure forms are nonnegative.
The term $a_tb^2\omega_0^{k+2}$ makes the pairing strictly positive, proving the required invertibility.

The path ends at $\omega_0^k$.
Append the path $((1-u)\omega_0+u\omega)^k$, $0\leqslant u\leqslant1$, to $\omega^k$.
Timorin's linear mixed Hodge--Riemann theorem \cite{Timorin} guarantees all required isomorphisms on this final path.
\end{proof}

\subsection{Extending a kernel to higher dimension}

\begin{lemma}\label{lem:dimension}
Suppose a Nakano positive $k\times k$ matrix $M$ on $\C^N$ has $0\ne\alpha\in\ker L_{\det M}^{p,q}$, where $p+q=N-k$.
On $\C^n$, $n>N$, put
\[
 \widetilde M=M+I_k\sum_{j=N+1}^nV_j,\qquad
 \widetilde\alpha=\alpha\wedge d\bar z_{N+1}\wedge\cdots\wedge d\bar z_n.
\]
Then $\widetilde M$ is Nakano positive and $0\ne\widetilde\alpha\in\ker L_{\det\widetilde M}^{p,q+n-N}$.
The construction preserves SD and, if present, dual Nakano positivity.
\end{lemma}
\begin{proof}
The Nakano tensor acquires positive identity blocks, as does its partial transpose.
Every new determinant term contains a $V_j$ annihilated by $\widetilde\alpha$.
The displayed expression also preserves SD.
\end{proof}

\section{Counterexamples in bidegree \texorpdfstring{$(1,n-k-1)$}{(1,n-k-1)}}\label{sec:unrestricted}

\begin{proof}[Proof of Theorem~\ref{thm:main}]
Using the coordinate forms fixed in Section~\ref{sec:positivity}, put
\[
 \mathcal N=\begin{pmatrix}
 V_1+V_3&E_{12}+E_{34}\\
 E_{21}+E_{43}&V_2+V_4
 \end{pmatrix},
 \qquad
 M_t=\diag(\omega_0 I_2+t\mathcal N,\omega_0 I_{k-2}),
 \qquad t\geqslant0.
\]
The last block is omitted when $k=2$.
If $e_a$ and $f_j$ are the standard bases in the matrix and coordinate directions, respectively, the Nakano matrix of $M_t$ is the positive definite matrix
\[
 I_{kn}+t(w_1w_1^*+w_2w_2^*),
 \qquad
 w_1=e_1\otimes f_1+e_2\otimes f_2,
 \quad
 w_2=e_1\otimes f_3+e_2\otimes f_4.
\]
Thus Nakano positivity holds for every $t\geqslant0$.

We first take $n=k+2$ and split the coordinate space as $\C^{k+2}=U\oplus W$, where $U$ consists of the first four coordinate directions and $W$ of the remaining $k-2$.
Set
\[
 \omega_U=V_1+V_2+V_3+V_4,
 \qquad \omega_W=\sum_{j=5}^{k+2}V_j,
 \quad \Delta=\det\mathcal N,
 \qquad \gamma=E_{13}-E_{31}+E_{24}-E_{42}.
\]
The kernel follows from two exterior-algebra identities:
\begin{equation}\label{eq:explicit-skew-identities}
 \gamma\wedge\Delta=-\frac12\gamma\wedge\omega_U^2,
 \qquad \gamma\wedge\omega_U^3=0.
\end{equation}
For verification, expand
\[
 \Delta=2V_1V_2+2V_3V_4+V_1V_4+V_2V_3
       -E_{12}E_{43}-E_{34}E_{21},
\]
where products denote wedges.
Both sides of the first identity are
\[
 -(E_{13}-E_{31})V_2V_4-(E_{24}-E_{42})V_1V_3.
\]
Since $\tr\mathcal N=\omega_U$, the determinant is
\[
 \det M_t=\omega_0^k+t\omega_U\wedge\omega_0^{k-1}
                      +t^2\Delta\wedge\omega_0^{k-2}.
\]
Expand $\omega_0=\omega_U+\omega_W$.
Only terms with $U$-degree $(2,2)$ survive multiplication by $\gamma$: smaller $U$-degree would exceed the dimension of $W$.
For larger $U$-degree, $\gamma\wedge\omega_U^3=0$ and $\gamma\wedge\Delta\wedge\omega_U =-\tfrac12\gamma\wedge\omega_U^3=0$.
Hence
\begin{equation}\label{eq:explicit-unrestricted-kernel}
 \gamma\wedge\det M_t
 =\left(\binom{k}{2}+(k-1)t-\frac{t^2}{2}\right)
   \gamma\wedge\omega_U^2\wedge\omega_W^{k-2}.
\end{equation}
The scalar vanishes at the explicit positive parameter
\[
 t_k=k-1+\sqrt{(k-1)(2k-1)}.
\]
Since $\gamma\ne0$, this proves that $L_{\det M_{t_k}}^{1,1}$ has a nonzero kernel on $\C^{k+2}$.

For $n>k+2$, apply Lemma~\ref{lem:dimension}: use the same matrix $M_{t_k}$ with the full form $\omega_0=\sum_{j=1}^nV_j$, and put
\[
 \widetilde\gamma
 =\gamma\wedge d\bar z_{k+3}\wedge\cdots\wedge d\bar z_n.
\]
Every determinant term involving a new coordinate contains some $V_j$ with $j>k+2$ and therefore vanishes after multiplication by $\widetilde\gamma$.
Equation~\eqref{eq:explicit-unrestricted-kernel} then gives $\widetilde\gamma\wedge\det M_{t_k}=0$.
Its bidegree is $(1,n-k-1)$, as required.
Finally, Nakano positivity implies Griffiths positivity, while a singular Lefschetz map excludes membership in the Hodge--Riemann cone.
\end{proof}

The failure of dual Nakano positivity can be seen directly from the spectrum.
The partial transpose of the Nakano matrix has eigenvalues $1+t$ with multiplicity six, $1-t$ with multiplicity two, and $1$ with multiplicity $kn-8$.
Indeed, each of the two rank-one terms $w_jw_j^*$ has, on its four-dimensional tensor block, partial transpose with eigenvalues $1,1,1,-1$.
Since $t_k\geqslant1+\sqrt3>1$, the matrix $M_{t_k}$ is not dual Nakano positive.
Its transpose is dual Nakano positive but is not Nakano positive, and has the same determinant.
Hence neither matrix answers Question~\ref{qu:simultaneous-nakano} when $n=4$ and $k=2$.

\section{Dually Lorentzian polynomials and the SD condition}
\label{sec:sdpositive}

Throughout this section, $M=\sum_jB_jV_j$, where every $B_j$ is positive definite, and $\Omega=\det M=\sum_{|I|=k}d_IV_I$ as in Section~\ref{sec:preliminaries}.
We first identify the blocks of the Lefschetz maps, then use an auxiliary dually Lorentzian polynomial and the generalized Alexandrov--Fenchel inequality to determine the signature of the blocks occurring in bidegree $(1,n-k-1)$.

\subsection{A decomposition of the Lefschetz map}

For an integer $r\geqslant0$ and $J\subset[n]$ with $|J|=k+2r$, define a matrix indexed by the $r$-subsets of $J$:
\[
 \HH_r^J=\left(
 \begin{cases}
 d_{J\setminus(R\cup S)},&R\cap S=\varnothing,\\
 0,&R\cap S\ne\varnothing
 \end{cases}\right)_{R,S\in\binom Jr}.
\]
In particular, $\HH_0^J=(d_J)$.

\begin{proposition}\label{prop:sd-blocks-new}
For $p+q=n-k$, the map $L_\Omega^{p,q}$ is a direct sum of the matrices $\HH_r^J$, with $0\leq r\leq\min(p,q)$.
A block $\HH_r^J$ occurs once for each ordered partition
\[
 [n]\setminus J=H\sqcup K,
 \qquad |H|=p-r,\quad |K|=q-r.
\]
Consequently, $L_\Omega^{p,q}$ is invertible if and only if all these matrices are nonsingular.
\end{proposition}

\begin{proof}
Fix disjoint $H,K$ and set $J=[n]\setminus(H\cup K)$.
Use the source and target bases
\[
 \beta_S=dz_H\wedge d\bar z_K\wedge V_S,
 \qquad
 \delta_R=dz_H\wedge d\bar z_K\wedge V_{J\setminus R},
 \qquad R,S\in\binom Jr.
\]
The term $d_IV_I$ sends $\beta_S$ to $d_I\delta_R$ precisely when $I=J\setminus(R\cup S)$ and $R\cap S=\varnothing$.
Since the $V_i$ have even degree, no reordering sign occurs.
These source blocks partition a basis of $V^{p,q}$: for a coordinate form indexed by $(I_+,I_-)$, take $H=I_+\setminus I_-$, $K=I_-\setminus I_+$ and $S=I_+\cap I_-$.
The same construction partitions the target basis.
\end{proof}

\subsection{An auxiliary dually Lorentzian polynomial}

We recall only the polynomial definitions needed to apply the inequality of Ross, S\"u\ss, and Wannerer.
The notion of a Lorentzian polynomial is due to Br\"and\'en and Huh \cite[Definition~2.1]{BH}.
For degree $e\geqslant2$, these are the coefficientwise limits of homogeneous polynomials with all coefficients strictly positive such that, for every multi-index $\alpha$ with $|\alpha|=e-2$, the Hessian of the partial derivative $\partial^\alpha f$ is nonsingular and has exactly one positive eigenvalue.
In degrees zero and one, the Lorentzian polynomials are the homogeneous polynomials with nonnegative coefficients.

A polynomial is multiaffine if it has degree at most one in each variable.
For a homogeneous multiaffine polynomial $s$ in $\ell$ variables, define
\[
 s^\vee(y)=y_1\cdots y_{\ell}s(y_1^{-1},\ldots,y_{\ell}^{-1}).
\]
Ross, S\"u\ss, and Wannerer call $s$ \emph{dually Lorentzian} when $s^\vee$ is Lorentzian \cite[Definition~4.2]{RSW}; their normalization operator is the identity in this multiaffine case.

A real polynomial is stable, in the convention recalled by Br\"and\'en and Huh \cite[Section~2.1]{BH}, if it is zero or is nonvanishing when all its variables lie in the open upper half-plane.
Br\"and\'en and Huh proved that a homogeneous stable polynomial with nonnegative coefficients is Lorentzian \cite[Proposition~2.2]{BH}.
This criterion is the only Lorentzian-polynomial test required in the proof below.

\begin{lemma}\label{lem:auxiliary-polynomial}
For any positive definite Hermitian $k\times k$ matrices $B_1,\ldots,B_m$, there exist vectors $u_1,\ldots,u_{\ell}$ spanning $\C^k$ and real numbers $a_{\rho i}>0$ such that
\[
 B_i=\sum_{\rho=1}^\ell a_{\rho i}u_\rho u_\rho^*.
\]
For $U=(u_1,\ldots,u_{\ell})$, the polynomial
\[
 s(y)=\det\bigl(U\diag(y)U^*\bigr)
      =\sum_{|T|=k}|\det U_T|^2y_T
\]
is nonzero, homogeneous, multiaffine, and dually Lorentzian.
Here $U_T$ is the $k\times k$ submatrix formed by the columns indexed by $T\subset[\ell]$, in increasing order, and $y_T=\prod_{\rho\in T}y_\rho$.
\end{lemma}

\begin{proof}
Let $e_1,\ldots,e_k$ be the standard basis of $\C^k$.
The rank-one matrices formed from $e_a$, $e_a\pm e_b$ and $e_a\pm\ii e_b$ span the real space of Hermitian matrices, and their sum is a positive scalar multiple of the identity.
Their cone therefore contains the identity in its interior.
Congruence by $B_i^{1/2}$ gives a finitely generated rank-one cone containing $B_i$ in its interior.
Take the union of these generators, and let $G$ be their sum.
Each $B_i-\varepsilon_iG$ remains in the resulting cone for sufficiently small $\varepsilon_i>0$; adding back $\varepsilon_iG$ gives the stated refinement with all coefficients positive.

The classical Cauchy--Binet formula gives the expansion; see Horn and Johnson \cite[Section~0.8.7]{HornJohnson}.
The vectors $u_\rho$ span $\C^k$, so $s$ is nonzero.
If all $y_\rho$ lie in the upper half-plane, the matrix inside the determinant has positive definite imaginary part; hence $s$ is real stable.
The multiaffine dual is also real stable: inversion sends the upper half-plane to the lower half-plane, where $s$ is nonvanishing by conjugation.
Both polynomials have nonnegative coefficients.
By Br\"and\'en and Huh \cite[Proposition~2.2]{BH}, $s^\vee$ is Lorentzian, so $s$ is dually Lorentzian.

\end{proof}

\subsection{The generalized Alexandrov--Fenchel inequality and the first blocks}

\begin{proposition}[Ross--S\"u\ss--Wannerer {\cite[Theorem~8.2]{RSW}}]\label{prop:AF}
Let $s\ne0$ be a multiaffine dually Lorentzian polynomial of degree $m-2$ in $\ell$ variables, and let $A_1,\ldots,A_{\ell},Q$ be positive definite real symmetric $m\times m$ matrices.
Define the symmetric bilinear form on real symmetric matrices $P_1,P_2$ by
\[
 \mathcal B(P_1,P_2)=D_m(P_1,P_2,s(A_1,\ldots,A_{\ell})),
\]
where the notation is extended linearly over the monomials of $s$.
Then, for every real symmetric $P$,
\[
 \mathcal B(P,Q)^2\geqslant
 \mathcal B(P,P)\mathcal B(Q,Q),
\]
with equality if and only if $P$ is proportional to $Q$.
\end{proposition}

This generalized Alexandrov--Fenchel inequality extends Alexandrov's mixed-discriminant inequality \cite{Alexandrov}.
We now derive the matrix statement needed for the SD condition.

\begin{proposition}
\label{lem:strict-deletion-new}
For every $J\subset[n]$ with $|J|=k+2$, the matrix $\HH_1^J$ has signature $(1,k+1)$.
\end{proposition}

\begin{proof}
Relabel $J=[m]$, where $m=k+2$, and use the representation and the dually Lorentzian polynomial $s$ from Lemma~\ref{lem:auxiliary-polynomial}.
Put
\[
 A_\rho=\diag(a_{\rho1},\ldots,a_{\rho m}).
\]
The matrix $A_\rho$ is positive definite.
Apply Proposition~\ref{prop:AF} to $s$ and the matrices $A_\rho$.
Writing $X(x)=\diag(x_1,\ldots,x_m)$, multilinearity gives
\[
 m!\mathcal B(X(x),X(y))=x^{\mathsf T}\HH_1^Jy.
\]
Indeed, the rank-one representation gives
\[
 \det\Bigl(\sum_i x_iB_i\Bigr)
 =s\Bigl(\sum_i a_{1i}x_i,\ldots,\sum_i a_{\ell i}x_i\Bigr).
\]
For $I\subset[m]$ with $|I|=k$, the coefficient of the squarefree monomial $x_I$ is $d_I$.
The two remaining diagonal positions give the corresponding off-diagonal entry of $\HH_1^J$.
Since all $d_I>0$, we have $\mathcal B(I_m,I_m)>0$.
The equality statement makes $\mathcal B$ strictly negative on the hyperplane $\mathcal B(X,I_m)=0$ in the diagonal subspace: a nonzero $X$ on this hyperplane cannot be proportional to $I_m$, so equality in Proposition~\ref{prop:AF} is excluded.
This hyperplane has dimension $m-1$, proving the signature $(1,m-1)$.
\end{proof}

The strictness in this argument is essential.
A Lorentzian quadratic can have positive off-diagonal coefficients and a singular Hessian: for example, $2x_1x_2+2x_3x_4+(x_1+x_2)(x_3+x_4)$ has Hessian eigenvalues $4,0,-2,-2$.
The corresponding form is a strictly strongly positive member of the family in Berndtsson--Sibony \cite[Section~9]{BS}, but
\[
 (V_1+V_2-V_3-V_4)\wedge
 \left(V_1V_2+V_3V_4+\tfrac12(V_1+\cdots+V_4)^2\right)=0.
\]
It therefore fails both the hard Lefschetz theorem and the Hodge--Riemann relations in bidegree $(1,1)$.
Nondegeneracy here follows from positive definite determinantal data and the equality case in Proposition~\ref{prop:AF}.

\subsection{The Lefschetz isomorphism and proof of Theorem~\ref{thm:sd}}

\begin{proof}[Proof of Theorem~\ref{thm:sd}]
By Proposition~\ref{prop:sd-blocks-new}, the only blocks of $L_\Omega^{1,n-k-1}$ are $\HH_0^J$ and $\HH_1^J$.
The former are the positive scalars $d_J$; the latter are nonsingular by Proposition~\ref{lem:strict-deletion-new}.
Hence $L_\Omega^{1,n-k-1}$ is an isomorphism for every positive SD matrix.
Applying this conclusion along the path in Lemma~\ref{lem:sd-reduction} proves the Hodge--Riemann assertion for any positive reference form.
\end{proof}

When $k=2$, this recovers Chen's affirmative results for $n=4,5$.
In dimension four, if $d_{ij}=d_{\{i,j\}}$ is the coefficient of $V_iV_j$, set
\[
 \ell_1=\sqrt{d_{12}d_{34}},\quad
 \ell_2=\sqrt{d_{13}d_{24}},\quad
 \ell_3=\sqrt{d_{14}d_{23}}.
\]
Then
\[
 \det\HH_1^{[4]}
 =-(\ell_1+\ell_2+\ell_3)(\ell_1+\ell_2-\ell_3)
   (\ell_1-\ell_2+\ell_3)(-\ell_1+\ell_2+\ell_3).
\]
Chen proved the required strict triangle inequalities using Ptolemy's inequality \cite[proof of Theorem~4.6]{Chen}.
The preceding application of Ross--S\"u\ss--Wannerer's inequality provides the corresponding signature statement in every rank.

\section{SD counterexamples in bidegree \texorpdfstring{$(2,n-k-2)$}{(2,n-k-2)}}\label{sec:sdnegative}

\subsection{The remaining blocks in bidegree \texorpdfstring{$(2,n-k-2)$}{(2,n-k-2)}}

Theorem~\ref{thm:sd} controls the blocks with $r=0,1$ in all dimensions.
The next bidegree introduces the matrices $\HH_2^J$.

\begin{proposition}\label{prop:second-blocks}
Let $n\geqslant6$, $2\leqslant k\leqslant n-4$, and let $M=\sum_jB_jV_j$, where every $B_j$ is positive definite.
For $\Omega=\det M$,
\[
 L_\Omega^{2,n-k-2}\text{ is an isomorphism}
 \quad\Longleftrightarrow\quad
 \det\HH_2^J\ne0\quad\text{for every }J\subset[n],\ |J|=k+4.
\]
\end{proposition}
\begin{proof}
Proposition~\ref{prop:sd-blocks-new} gives only the blocks with $r=0,1,2$.
The first two are nonsingular by positivity of $d_J$ and Proposition~\ref{lem:strict-deletion-new}.
Every block $\HH_2^J$ with $|J|=k+4$ occurs, giving the equivalence.
\end{proof}

This is a criterion for the Lefschetz isomorphism; the Hodge--Riemann property additionally requires the appropriate component in the cone definition.
For Theorem~\ref{thm:sdnegative} it suffices to make one of these blocks singular.
We do this first in dimension six and then increase the rank and dimension.

\subsection{The Ross--Toma seed and a rank-two determinant realization}
The degeneracy underlying the seed case first appeared in Ross--Toma \cite[Example~9.2, p.~37 of arXiv:1905.13636v3]{RT}.
Their example uses the second Chern class of an ample real twist of a rank-three bundle on $(\mathbf P^2)^3$.
Here we present a positive rank-two SD determinant realization of the same form, up to positive scaling; the comparison is made explicitly below.
On $\C^6$, put $W_j=V_{2j-1}+V_{2j}$ for $j=1,2,3$, and set
\[
 S_1=\begin{pmatrix}1&0\\0&-1\end{pmatrix},\qquad
 S_2=\frac12\begin{pmatrix}-1&\sqrt3\\\sqrt3&1\end{pmatrix},\qquad
 S_3=\frac12\begin{pmatrix}-1&-\sqrt3\\-\sqrt3&1\end{pmatrix}.
\]
Each $S_j$ has eigenvalues $1,-1$.
Hence
\begin{equation}\label{eq:newseed}
 \widehat M_t=\sum_{j=1}^3(I_2+\sqrt t\,S_j)W_j,
 \qquad 0\leqslant t\leqslant\tfrac12,
\end{equation}
is Griffiths positive and satisfies SD.
Write
\[
 \Sigma=\sum_{j=1}^3W_j^2,\qquad
 \Pi=\sum_{i<j}W_iW_j.
\]
Expansion of the $2\times2$ determinant gives
\begin{equation}\label{eq:newomega}
 \Omega_t:=\det \widehat M_t=(1-t)\Sigma+(2+t)\Pi.
\end{equation}
Since $W_j^3=0$, the products satisfy
\[
 \Sigma^2=2\sum_{i<j}W_i^2W_j^2,\quad
 \Sigma\Pi=W_1^2W_2W_3+W_1W_2^2W_3+W_1W_2W_3^2,
 \quad \Pi^2=\tfrac12\Sigma^2+2\Sigma\Pi.
\]
For $u=(2+t)/(2(1-t))$, it follows that
\begin{equation}\label{eq:explicitSDkernel}
 \Omega_t(\Pi-u\Sigma)=(1-t)(1+4u-2u^2)\Sigma\Pi.
\end{equation}
At
\[
 t_* =\sqrt6-2\in(0,\tfrac12),\qquad u_*=1+\frac{\sqrt6}{2},
\]
the right side is zero.
Since $\Pi-u_*\Sigma\ne0$, this is an explicit kernel vector for $L_{\Omega_{t_*}}^{2,2}$.

To identify the source of this seed, put $X=(\mathbf P^2)^3$ and let $h_j$ be the pullback of the hyperplane class from the $j$th factor.
Ross--Toma take $E=\bigoplus_{j=1}^3\pi_j^*\mathcal O_{\mathbf P^2}(1)$, where $\pi_j$ is the projection to that factor, and its real twist $E\langle\tau h\rangle$ by $h=h_1+h_2+h_3$.
Their formula is
\[
 c_2(E\langle\tau h\rangle)
 =\sum_{i<j}h_ih_j+(2\tau+3\tau^2)(h_1+h_2+h_3)^2.
\]
The relations $h_j^3=0$ agree with $W_j^3=0$.
Under the identification $h_j\leftrightarrow W_j$, for $t>0$,
\[
 \Omega_t=3t\bigl[\Pi+c(\Sigma+2\Pi)\bigr],
 \qquad c=\frac{1-t}{3t}.
\]
At $t=t_*$, one has $c=1/\sqrt6$.
Taking the positive solution of $2\tau+3\tau^2=1/\sqrt6$ gives precisely Ross--Toma's degeneration, with the kernel displayed above.
Thus the seed form is due to Ross--Toma; the rank-two determinant realization and the extension below are the constructions used here.

For use in the extension, let $\mathsf K_0(t)$ be the block on the basis $(V_S)_{S\in\binom{[6]}2}$, with target basis $(V_{[6]\setminus R})_{R\in\binom{[6]}2}$.
The coefficients of $\Omega_t$ are
\begin{equation}\label{eq:symmetricq}
 d_{ij}(t)=\begin{cases}2(1-t),&i,j\text{ belong to the same pair},\\
 2+t,&i,j\text{ belong to different pairs}.
 \end{cases}
\end{equation}
Here $d_{\{i,j\}}(t)=d_{ij}(t)$ for $i<j$, consistently with the coefficient notation $d_I$ in Section~\ref{sec:preliminaries}.
Thus $(\mathsf K_0)_{R,S}=d_{[6]\setminus(R\cup S)}$ for disjoint pairs, and is zero otherwise.
Exact expansion gives the short factorization
\begin{equation}\label{eq:simpledet}
 \det \mathsf K_0(t)=2^{15}3^6(1-t)^3(2t+1)^3
 (t^2+t+1)^2(t^2+4t-2).
\end{equation}
In particular, its endpoint signs on $[0,\tfrac12]$ are opposite, and $t_*$ is its unique singular parameter on this interval.

\subsection{Increasing the rank}
We first add scalar diagonal entries to the rank-two example.
The main point is to preserve the opposite signs of the two endpoint determinants.
A small perturbation will then give positive definite coefficient matrices in the new coordinate directions.

For $m\geqslant0$, put $N=m+6$ and $\omega_{0,N}=\sum_{j=1}^NV_j$, the standard form on $\C^N$.
The middle block $\mathsf K_m(t)$ of $\Omega_t\omega_{0,N}^m/m!$, with source basis $(V_S)$ and target basis $(V_{[N]\setminus R})$ for $R,S\in\binom{[N]}2$, has entries
\[
 (\mathsf K_m)_{R,S}=\begin{cases}
 \displaystyle\sum_{\substack{1\leqslant a<b\leqslant6\\
 \{a,b\}\cap(R\cup S)=\varnothing}}d_{ab}(t),&R\cap S=\varnothing,\\
 0,&R\cap S\ne\varnothing,
 \end{cases}
 \qquad R,S\in\binom{[N]}2.
\]
Multiplication by $\omega_{0,N}$ from the span of the $V_i$ to the span of the $V_R$, $|R|=2$, has matrix $L_m$ with $(L_m)_{R,i}=1$ if $i\in R$ and $0$ otherwise.
The identity
\[
 \frac{\Omega_t\omega_{0,N+1}^{m+1}}{(m+1)!}
 =\frac{\Omega_t\omega_{0,N}^m}{m!}\wedge
 \left(V_{N+1}+\frac{\omega_{0,N}}{m+1}\right)
\]
explains the block decomposition below: list the old pairs first and the pairs containing $N+1$ last.
\begin{equation}\label{eq:suspension}
 \mathsf K_{m+1}=\begin{pmatrix}
 \mathsf K_m&\mathsf K_mL_m/(m+1)\\
 L_m^{\mathsf T}\mathsf K_m/(m+1)&0
 \end{pmatrix},\qquad
 L_m^{\mathsf T}\mathsf K_mL_m=(m+1)(m+2)R_m,
\end{equation}
where $(R_m)_{ii}=0$ and
\[
 (R_m)_{ij}=\sum_{\substack{1\leqslant a<b\leqslant6\\a,b\notin\{i,j\}}}
 d_{ab}(t)\quad(i\ne j).
\]
Both identities follow by counting each surviving coefficient: there are $m+1$ choices in the first and $(m+1)(m+2)$ in the second.

We check that $R_m$ has signature $(1,N-1)$ for the whole interval.
Its off-diagonal entries are $12$ within an original pair, $12+3t$ between original pairs, $4(5+t)$ between an original and a new index, and $6(5+t)$ between two new indices.
Differences within the three original pairs give eigenvalue $-12$ with multiplicity three.
Vectors constant on each pair with the three values summing to zero give $-6(2+t)$ with multiplicity two.
New-coordinate differences give $-6(5+t)$ with multiplicity $m-1$ when $m\geqslant1$.
For $m\geqslant1$, the remaining invariant subspace consists of vectors constant on the six original indices and constant on the $m$ new indices.
In the basis of these two indicator vectors its matrix is
\[
 2(5+t)\begin{pmatrix}6&2m\\12&3(m-1)\end{pmatrix},
\]
whose determinant is $-24(m+3)(5+t)^2<0$.
Although this basis is not orthonormal, the matrix represents the restriction of the real symmetric matrix $R_m$ and has real eigenvalues.
It therefore has one eigenvalue of each sign.
For $m=0$, the remaining space is one-dimensional with eigenvalue $12(5+t)>0$.
This proves the signature assertion.

To propagate the endpoint signs, we eliminate the off-diagonal blocks in~\eqref{eq:suspension} by the explicit change of variables
\[
 \begin{pmatrix}I&0\\-L_m^{\mathsf T}/(m+1)&I\end{pmatrix}
 \mathsf K_{m+1}
 \begin{pmatrix}I&-L_m/(m+1)\\0&I\end{pmatrix}
 =\begin{pmatrix}\mathsf K_m&0\\0&-\dfrac{m+2}{m+1}R_m\end{pmatrix}.
\]
The identity blocks have sizes $\binom N2$ and $N$, respectively.
Both change-of-variable matrices have determinant one, so
\[
 \det \mathsf K_{m+1}(t)=\det \mathsf K_m(t)\,
 \det\left(-\frac{m+2}{m+1}R_m(t)\right).
\]
Since $R_m$ has one positive and $N-1$ negative eigenvalues, the second determinant is negative.
Thus each increase of $m$ reverses both endpoint signs.
Starting from~\eqref{eq:simpledet}, induction gives
\[
 \operatorname{sgn}\det \mathsf K_m(0)=(-1)^{m+1},
 \qquad
 \operatorname{sgn}\det \mathsf K_m(\tfrac12)=(-1)^m.
\]

For $k=m+2\geqslant3$, introduce strict positivity in the new directions:
\[
 \mathcal M_{t,\varepsilon}=
 \begin{pmatrix}
 \widehat M_t+\varepsilon\omega_{\mathrm{new}}I_2&0\\0&\omega_{0,N}I_m
 \end{pmatrix},\qquad
 \omega_{\mathrm{new}}=\sum_{j=7}^NV_j,\quad\varepsilon>0.
\]
Its SD coefficient matrices are positive definite.
Its middle block converges to $m!\mathsf K_m(t)$ as $\varepsilon\to0$.
For sufficiently small positive $\varepsilon$, the endpoint determinants retain opposite signs.
Fix one such $\varepsilon$;
the determinant of the middle block is continuous in $t\in[0,\tfrac12]$, so an intermediate parameter gives a singular map in dimension $N=k+4$.
Applying Lemma~\ref{lem:dimension} gives $L_{\det M}^{2,n-k-2}$ singular for every $n\geqslant6$ and $2\leqslant k\leqslant n-4$.
This proves Theorem~\ref{thm:sdnegative}.

\section{Open questions}\label{sec:scope}
The results of this paper concern pointwise forms.
For Hodge--Riemann results on top Chern classes of ample bundles on smooth projective varieties, see Lu--Zheng \cite[Theorem~5.1]{LZ} and Lin \cite[Theorem~1.5]{Lin}.
The examples here do not construct global ample bundles with prescribed curvature.

Let $M$ be a Hermitian $k\times k$ matrix of constant $(1,1)$-forms on $\C^n$, where $1\leqslant k\leqslant n$.
A real $(k,k)$-form is strictly weakly positive if its restriction to every complex $k$-plane is a positive volume form.
Finski's positivity theorem, together with the local realization of curvature tensors \cite[Theorem~1.1 and Proposition~2.10]{Finski}, gives
\[
\begin{gathered}
 M\text{ Nakano positive or dual Nakano positive}\\
 \Longrightarrow\quad\det M\text{ strictly weakly positive}.
\end{gathered}
\]
When $n-k=0$ or $1$,
\[
\begin{gathered}
 \det M\text{ belongs to the Hodge--Riemann cone}\\
 \Longleftrightarrow\quad\det M\text{ is strictly weakly positive}.
\end{gathered}
\]
Indeed, the relevant condition is positivity of a scalar when $k=n$ and of a Hermitian form on one-forms when $k=n-1$; the positive cone in each case is convex.
Thus Dinh--Nguy\^en's question has an affirmative answer in these boundary cases under either Nakano positivity or dual Nakano positivity.

For general Griffiths positive matrices, strict weak positivity of the determinant is known for $k\leqslant3$: the rank-one case is immediate, Griffiths proved the rank-two case \cite{Griffiths}, and Wan proved the rank-three case \cite[Theorem~0.3]{Wan}.

\begin{Question}\label{qu:griffiths-determinant}
Let $n\geqslant k\geqslant4$ and let $M$ be a Griffiths positive Hermitian $k\times k$ matrix of constant $(1,1)$-forms on $\C^n$.
Is $\det M$ strictly weakly positive?
\end{Question}

Restriction to complex $k$-planes reduces Question~\ref{qu:griffiths-determinant} to $n=k$, where $\det M$ is a top-degree form.
By local realization of curvature tensors, this is the top-Chern case of Griffiths' positivity question \cite[p.~247]{Griffiths}, equivalently Finski's double mixed discriminant problem in dimension $k$ \cite[Propositions~2.10--2.11 and Theorem~3.4]{Finski}.
It determines the remaining cases of Dinh--Nguy\^en's question:
\[
 n=k\geqslant4\qquad\text{and}\qquad n=k+1\geqslant5.
\]
Any counterexample in these cases must be neither Nakano positive nor dual Nakano positive.
Theorem~\ref{thm:main} concerns $n-k\geqslant2$: it shows that strict weak positivity of the determinant does not suffice for the Hodge--Riemann property.

The examples in Theorem~\ref{thm:main} and their transposes each satisfy only one of the two Nakano positivity conditions.
This leaves the following question.

\begin{Question}\label{qu:simultaneous-nakano}
Does there exist a Hermitian $2\times2$ matrix $M$ of constant $(1,1)$-forms on $\C^4$ that is both Nakano positive and dual Nakano positive and for which $L_{\det M}^{1,1}$ is singular?
\end{Question}

By Theorem~\ref{thm:sd}, any such matrix must fail the SD condition.

\appendix
\section{Python verification programs}\label{sec:verification}
Two self-contained Python verification files are available from the author's \href{https://zhangchi-chen.github.io/}{homepage} via the links below.
The programs require Python~3 and SymPy and use exact arithmetic:
\begin{enumerate}
\item \href{https://zhangchi-chen.github.io/downloads/hrr/verify_unrestricted.py}{\nolinkurl{https://zhangchi-chen.github.io/downloads/hrr/verify_unrestricted.py}}
checks the exterior identities and determinant calculation for the seed case $(k,n)=(2,4)$ in Section~\ref{sec:unrestricted}.
It also checks the determinant and Nakano coefficient matrices for $k=3,4,5$ and the scalar identity defining $t_k$ for arbitrary rank.
\item \href{https://zhangchi-chen.github.io/downloads/hrr/verify_sd.py}{\nolinkurl{https://zhangchi-chen.github.io/downloads/hrr/verify_sd.py}}
checks the SD seed case $(k,n)=(2,6)$ in Section~\ref{sec:sdnegative}: the determinant formula and factorization and the explicit kernel at $t_* = \sqrt6-2$.
For the factorization it checks sixteen distinct integer values; both sides have degree at most fifteen.
It also checks the rank-extension block identities at three successive sizes and the scalar identities used in the general argument.
\end{enumerate}

To run the downloaded files locally, use the following commands in their directory:
\begin{verbatim}
python3 -m pip install sympy
python3 verify_unrestricted.py
python3 verify_sd.py
\end{verbatim}
Alternatively, one can also verify the computation online for free at \href{https://sagecell.sagemath.org/}{\nolinkurl{https://sagecell.sagemath.org/}} by selecting \texttt{Python} in the language menu, pasting the complete contents of one downloaded file into the input cell, and clicking \texttt{Evaluate}.
Each file should finish with \texttt{ALL CHECKS PASSED.}

\bigskip
\noindent\textbf{Acknowledgements.}
The author thanks Siarhei Finski for discussing problems in linear algebra during the conference ``Recent Developments in Complex Geometry'' in B\k{e}dlewo, Poland, in 2026, and for later discussions of the cases $n=k$ and $n=k+1$.
These discussions renewed the author's interest in Dinh--Nguy\^en's question, and lead to the formalism of the open questions listed in this article.
The author thanks Jian Xiao for suggesting the connection between the construction in Theorem~\ref{thm:sdnegative} and Ross--Toma's work \cite[Example~9.2, p.~37 of arXiv:1905.13636v3]{RT}.
The author also thanks Weizhe Zheng and Enhan Li for inspiring discussions and remarks.

{\raggedright
Zhangchi Chen was supported by the National Key R\&D Program of China (No.~2025YFA1018300), the National Natural Science Foundation of China (grant No.~12501104), the Science and Technology Commission of Shanghai Municipality (grant No.~22DZ2229014), the Shanghai Sailing Program (grant No.~24YF2709900), and the Shanghai Pujiang Program (grant No.~24PJA023).
He also acknowledges support from the CDP C2EMPI (R-CDP-24-004-$C^2$EMPI).
\par
}

\bigskip
\noindent\textbf{Statement on the use of artificial intelligence.}
The counterexamples in Theorem~\ref{thm:main} were constructed by AI.
Theorem~\ref{thm:sdnegative} was also obtained using AI; its underlying idea comes from Ross--Toma's construction \cite[Example~9.2, p.~37 of arXiv:1905.13636v3]{RT}, where the seed degeneration first appeared.
The examples presented here give a rank-two SD determinant realization and extensions in rank and dimension.
The author checked the seed calculations with SageMath.
OpenAI's Codex and GLM-5.3 served as interactive research, computation, and writing assistants.
The author has fully read, checked, and understood the manuscript and takes sole responsibility for it.

\end{document}